\documentclass[11pt]{amsart}
\usepackage[english]{babel}
\usepackage{amssymb,amsmath,amsthm}
\usepackage[latin2]{inputenc}
\usepackage{amsfonts, verbatim, amscd}
\usepackage[foot]{amsaddr}

\reversemarginpar
\newtheorem{theorem}{Theorem}

\newtheorem{lemma}[theorem]{Lemma}

\newtheorem{coro}[theorem]{Corollary}
\def\acknowledgment{\par\addvspace{17pt}\small\rmfamily
\trivlist\if!\ackname!\item[]\else
\item[\hskip\labelsep
{\bfseries\ackname}]\fi}

\def\C{\mathbb{C}}
\def\R{\mathbb{R}}

\def\D{\mathbb{D}}

\newcommand{\du}{\mathrm{d}}

\begin{document}
\title{On $J$-holomorphic variational vector fields and extremal discs}

\author{Uro\v s Kuzman}

\address[Uro\v s Kuzman]{Faculty of Mathematics and Physics, University of Ljubljana, 
and Institute of Mathematics, Physics and Mechanics, Jadranska 19, 
1000 Ljubljana, Slovenia, uros.kuzman@fmf.uni-lj.si}


%

\begin{abstract}
We prove that every $J$-holomorphic variational vector field can be realized as derivation $\frac{d}{dt}_{|t=0}f_t$ where $(f_t)$ is a one parametric family of $J$-holomorphic discs. Furthermore, we discuss properness of an extremal $J$-holomorphic disc in a bounded pseudoconvex domain.  
\end{abstract}
\maketitle

\section{Introduction}
Let $\D \subset \C$ be the unit disc and $J$ a smooth almost complex structure defined on $\R^{2n}$, that is, a matrix function satisfying $J^2 = -id$. We denote by $J_{st}$ the standard structure on $\R^{2n}$ corresponding to the multiplication by the imaginary unit and call a $\mathcal C^{1}$ map $f\colon \D\to \R^{2n}$ a $J$-holomorphic disc if for all $\zeta=x+iy \in \D$ we have 
\begin{equation}\label{JvR}
df(\zeta) \circ J_{st} = J(f(\zeta)) \circ df(\zeta) \Longleftrightarrow f_x+J(f)f_y=0.
\end{equation}
In this paper, we discuss variations of such discs. 

Let $(f_t)_{t\in\R}$ be a one parameter family of $J$-holomorphic discs in $\R^{2n}$. We denote by $f=f_0$. It follows from (\ref{JvR}) that the vector field
$$V(\zeta):=\frac{\mathrm{d}}{\mathrm{d}t}_{|_{t=0}}f_t(\zeta)$$
satisfies the following equation
\begin{equation}\label{VF}
V_x+J(f)V_y+\textrm{d}_fJ(V)f_y=0,
\end{equation}
where $\textrm{d}_fJ(V)$ denotes the linearization of $g\to J(g)$ at $f$.  We call a $\mathcal{C}^1$-smooth solution of (\ref{VF}) a \emph{$J$-holomorphic variational vector field along $f$}. 
We prove that every such vector field can be realized as a derivation of a one parametric family of $J$-holomorphic discs.  

\begin{theorem}
Assume that $\det(J+J_{st})\neq 0$. Let $f\in \mathcal{C}^{1,\alpha}(\overline{\D})$ be a $J$-holomorphic disc and let $V\in \mathcal{C}^{1,\alpha}(\overline{\D})$ be a $J$-holomorphic variational vector field along $f$. There exists a family of $J$-holomorphic discs $(f_t)\subset \mathcal{C}^{1,\alpha}(\overline{\D})$ such that $V=\frac{\mathrm{d}}{\mathrm{d}t}_{|_{t=0}}f_t$ and $f_0=f$.  
\end{theorem}

\noindent We prove this statement in \S 2 using non-linear techniques developped in \cite{BERKUZ} and linear theory from \cite{ST}.  In \S3 a refined version of this theorem is used in the following application concerning extremal discs. 

Let $\Omega$ be a bounded domain in $\R^{2n}$. A $J$-holomorphic disc from $\D$ to $\Omega$ is called \emph{extremal} if for every $J$-holomorphic disc $g:\D\to \Omega$ such that $g(0)=f(0)$ and $g'(0)=\lambda f'(0)$ for some $\lambda\geq 0$, it follows that $\lambda \leq 1$. We denote here by $f'$ the velocity vector field $f_x$ which, in the usual holomorphic case, agrees with the standard complex derivative. When $J\equiv J_{st}$ and $\Omega$ is strongly convex these discs were discussed in the celebrated paper of Lempert \cite{LEMPERT}. He proved that they are uniquely determined by the pair $(f(0),f'(0))$ and can be extended to a proper smooth embedding attached to the boundary $\partial\Omega$. However, the theory becomes less clear when dealing with pseudoconvex domains. For instance, Sibony's examples in \cite{PANG} show that extremal discs are not unique and that the relation between $f(\partial\D)$ and $\partial\Omega$ is less trivial. Nevertheless, using a variational approach, Poletsky showed that $\mathcal{C}^1$-closed extremal discs are proper when $\Omega$ is bounded by a $\mathcal{C}^2$-plurisubharmonic function \cite{POLETSKY} (see also \cite[Theorem 1]{BDSJoo}).

We provide a generalization of this last statement for the case of non-integrable structures. That is, we improve \cite[Proposition 4.1]{gau-joo}.

\begin{theorem}\label{main-thm}
Assume that $\det(J+J_{st})\neq 0$. Let $\Omega=\left\{\rho<0\right\}\subset \R^{2n}$ be a bounded domain defined by a $\mathcal{C}^2$-smooth $J$-plurisubharmonic function. Let $f\in \mathcal{C}^{2,\alpha}(\overline{\mathbb{D}})$ be an extremal $J$-holomorphic disc. Then $f(\partial\D) \subset \partial \Omega$.
\end{theorem}
\noindent When an extremal disc is proper, one can discuss its relation to the so-called stationary discs that were also introduced by Lempert in \cite{LEMPERT}. In the standard case, this refers to biholomorphic invariants that can be defined in two equivalent ways: as discs that admit a certain proper meromorhic lift to the cotangent bundle \cite{TUMANOV}; as stationarity solutions of the Euler-Lagrange equation \cite{POLETSKY,PANG,BDSJoo}. As explained in \cite{gau-joo}, in the non-integrable case we fancy the second approach, since the first one does not provide us with the necessarry conditions for a disc to be extremal. However, such theory of $J$-stationary discs was developped only for small perturbations of $J_{st}$. Using our Theorem 1 this can now be improved. We discuss this briefly at the end of the paper.

\section{The Implicit function theorem}
Let $f\colon \D\to \R^{2n}$ be a $J$-holomorphic disc. If $\det(J+J_{st})\neq 0$ along the image $f(\mathbb{D})$, the $J$-holomorphicity condition (\ref{JvR}) can be turned into a non-linear Cauchy-Riemann system
\begin{eqnarray}\label{A}f_{\bar{\zeta}}+A(f)\overline{f_{\zeta}}=0,\end{eqnarray}
where $A(z)(v)=(J(z)+J_{st})^{-1}(J(z)-J_{st})(\bar{v})$ is complex linear in $v \in \C^n$ for every $z\in \C^{n}$ and can be treated as a complex matrix function; see \cite{ST3}. In this notation, a variational $J$-holomorphic vector field corresponds to a solution of the following complex equation
\begin{equation}\label{VFA}
V_{\bar{\zeta}}+A(f)\overline{V_{\zeta}}+\textrm{d}_fA(V)\overline{f_{\zeta}}=0.
\end{equation}
Moreover, this equation can be rewritten into a linear condition
\begin{equation}\label{VFB}
V_{\bar{\zeta}}+A(f)\overline{V_{\zeta}}+B^f_1V+B^f_2\overline{V}=0
\end{equation} 
determined by smooth matrix functions $B^f_1$ and $B^f_2$, depending on $A$ and $f$.

Note that in the Theorem 1 we can assume that $f$ is an embedding. Otherwise we consider its graph $\tilde{f}$ that is a $J_{st}\otimes J$-holomorphic disc in the almost complex space $(\R^{2n+2},J_{st}\otimes J)$ and interpolate the variational vector field $\tilde{V}(\zeta):=(\zeta,V(\zeta)).$ Therefore, after a change of coordinates, we can assume that $A(f)=0$. Indeed, see e.g.\ the Appendix to \cite{IR}. This means that the linear condition (\ref{VFB}) can be associated with the theory of so-called generalized vectors (see e.g. \cite{BOJARSKI,ST}).

Let us turn now towards the proof of Theorem 1. Suppose that $J\equiv J_{st}$. Then $A(f)\equiv 0$ and, instead of (\ref{A}) and (\ref{VFB}), we have the usual holomorphicity condition for $f$ and $V$. Thus the required family of discs is
$$f_t=f+tV.$$ 
We will mimic this idea using the non-linear techniques for Banach spaces. That is, we rely on the following standard theorem (see e.g. \cite[p.298]{LANG}).

\begin{theorem}\label{implicit}
Let $X$ and $Y$ be Banach spaces and let $\mathcal{F}\colon \mathbb{B}_{X}(x_0,\epsilon)\to Y$ be a $\mathcal{C}^1$ map defined on some ball in $X$. Suppose there is $C>0$ such that: 

\textit{i)} Given $x\in\mathbb{B}_{X}(x_0,\epsilon)$ and $v\in Y$ the operator $d_{x}\mathcal{F}\colon X\to Y$ is surjective 

and the equation $d_{x}\mathcal{F}(u)=v$ admits a solution with $\left\|u\right\|_X\leq C\left\|v\right\|_Y$.

\textit{ii)} For any $x_1,x_2\in \mathbb{B}_{X}(x_0,\epsilon)$ we have $\left\|d_{x_1}\mathcal{F}-d_{x_2}\mathcal{F}\right\|\leq \frac{1}{2C}$.

\noindent Then the ball $\mathbb{B}_Y\left(\mathcal{F}(x_0),\frac{\epsilon}{2C}\right)$ lies in  $\mathcal{F}\left(\mathbb{B}_{X}(x_0,\epsilon)\right).$
\end{theorem}

Let us fix $0<\alpha<1$. The classical Cauchy-Green operator given by 
$$
T(f)(z)=-\frac{1}{\pi}\iint_{\D}\frac{f(\xi)}{\xi-z}\, \du x \, \du y (\xi)
$$
is continuous when mapping between H\"older spaces of vector functions $T\colon \mathcal{C}^{0,\alpha}(\overline{\D})\to \mathcal{C}^{1,\alpha}(\overline{\D})$. Thus, we can define a continuous operator 
\begin{equation}\label{OPF}\mathcal{F}(f)=f+T\left(A(f)\overline{f_\zeta}\right)\end{equation}
mapping the space $\mathcal{C}^{1,\alpha}(\overline{\mathbb{D}})$ to itself.
Moreover, $T$ solves the usual $\bar{\partial}$-equation since $(Tu)_{\bar{z}}=u$. Hence, given a $J$-holomorphic disc $f$ the map $\mathcal{F}(f)$ is a $J_{st}$-holomorphic vector function.  

As explained above, without loss of generality, the derivative of $\mathcal{F}$ at $f$ can be assumed to be equal to
$$
d_f\mathcal{F}(V)=V+T\left(B_1V+B_2\overline{V}\right).
$$
This singular integral operator is known to be Fredholm with a possibly non trivial cokernel \cite{ST}. Therefore, we have to linearly perturb $\mathcal{F}$ in order to meet requirements of Theorem \ref{implicit}. 

\begin{lemma}\label{lemma4}
There exists a locally invertible perturbation $\tilde{\mathcal{F}}$ of the operator given in (\ref{OPF}) that maps $J$-holomorphic disc to the usual holomorphic ones. In particular, there is $C>0$ such that for every $W\in \mathcal{C}^{1,\alpha}(\overline{\mathbb{D}})$ there exists $V\in \mathcal{C}^{1,\alpha}(\overline{\mathbb{D}})$ satisfying $d_f\tilde{\mathcal{F}}(V)=W$ and $\left\|V\right\|_{1,\alpha}\leq C \left\|W\right\|_{1,\alpha}.$ 
\end{lemma}
\begin{proof}
This statement is proved in \cite{ST}. For completeness, we include some details. We introduce the real inner product of vector functions:
$$\left\langle f,g\right\rangle=\sum_{j=1}^{n}\iint_{\mathbb{D}} f_j\bar{g}_j \mathrm{d}x\mathrm{d}y.$$     
Let $H\subset \mathcal{C}^{k,\alpha}(\overline{\D})$ be the set of (usual) holomorphic vector functions. Then 
$$H+\mathrm{Range}(d_f\mathcal{F})=\mathcal{C}^{1,\alpha}(\overline{\D}).$$ 
Indeed, let $V\in\mathrm{Range}(d_f\mathcal{F})^{\bot}=\ker d_f\mathcal{F}^*$. The adjoint map of $d_f\mathcal{F}$ equals 
$$d_f\mathcal{F}^{*}(V)=V-\overline{B_1^T T(\bar{V})}-B_2^{T}T(\bar{V}).$$
Hence the vector $W=\overline{T(\overline{V})}$ satisfies the equation 
$$W_{\bar{\zeta}}-\overline{B_1^T}W-B_2^{T}\overline{W}=0$$
and is a so-called generalized analytic vector. Moreover, if $V$ is also orthogonal to the space $H$, then $W$ vanishes on 
$\C\backslash\overline{\D}$. This implies that $W\equiv 0$ and thus $V\equiv 0$ (see \cite[Corollary 3.4]{ST}). 

Let $N=\dim\ker d_f\mathcal{F}$. There exist $h_1, 
h_2,\ldots,h_N\in H$ such that 
\begin{eqnarray*}\label{vsota}\mathrm{Span}_{\R}(h_1,h_2,\ldots,h_N)\oplus 
\mathrm{Range}(d_f\mathcal{F})=\mathcal{C}^{1,\alpha}(\overline{\D}).\end{eqnarray*}
Thus, if $V_1,\ldots V_n$ form a basis of $\ker d_f\mathcal{F}$, the operator given by 
\begin{equation}\label{OPFT}\tilde{\mathcal{F}}(f)=\mathcal{F}(f)+\sum_{j=1}^{N}\mathrm{Re}(f,V_j)h_j\end{equation}
still maps $J$-holomorphic discs into $J_{st}$-holomorphic ones and admits an invertible linearization. Hence for $C$ one can set the norm of $d_f\tilde{\mathcal{F}}^{-1}.$
\end{proof}

Let us prove that, close to a fixed disc $f$, the operator $\tilde{\mathcal{F}}$ satisfies $ii)$ from Theorem \ref{implicit}. The size of the neighborhood depends on the $(1,\alpha)$-norm of $f$.
\begin{lemma}\label{lemma5}
Given a $J$-holomorphic disc $f$ with $A(f)\equiv 0$ and $\left\|f\right\|_{1,\alpha}<C$ there is $\epsilon_0>0$ such that  $\left\|g-f\right\|_{1,\alpha}<\epsilon_0$ implies 
$\left\|d_{g}\tilde{\mathcal{F}}-d_{f}\tilde{\mathcal{F}}\right\|<\frac{1}{4C}.$ 
\end{lemma}
\begin{proof}
Since, in general $A(g)\neq 0$, we have 
$$d_{g}\tilde{\mathcal{F}}(V)-d_{f}\tilde{\mathcal{F}}(V)=I_1+I_2+I_3,$$ 
where 
$$
\left\{\begin{array}{lll} 
I_1&= & T\left(A(g)\overline{V_{\zeta}}\right),\\
I_2&=&T\left(B_1^g-B_1^{f})V+(B_2^g-B_2^f)\bar{V}\right).\\
I_3&=& \sum_{j=1}^{N}\mathrm{Re}(f-g,V_j)h_j
\end{array}
\right.$$                                                                                                                                 
Let us assume that $\left\|g-f\right\|_{1,\alpha}<1$. Then there exist $C_k>0$, $k\in\left\{1,2,3\right\},$ such that 
$$\left\|I_k\right\|_{0,\alpha}\leq C_k \left\|g-f\right\|_{1,\alpha}\left\|V\right\|_{1,\alpha}.$$
Indeed, the bound for $I_3$ is obvious. Furthermore, $T\colon\mathcal{C}^{0,\alpha}(\overline{\mathbb{D}})\to \mathcal{C}^{1,\alpha}(\overline{\mathbb{D}})$ is bounded, hence we only have to bound the $(0,\alpha)$-norm of its arguments in $I_1$ and $I_2$. In $I_1$ the required $\alpha$-bound depends on the coefficients of $A$ and $\left\|g\right\|_{1,\alpha}<1+C$. However, in $I_2$ we have  
$$B^f_1V+B^f_2\bar{V}= \left(\sum_{j=1}^n\frac{\partial A}{\partial z_j}(f)V_j+\frac{\partial A}{\partial \bar{z}_j}(f)\bar{V}_j\right)\overline{f_\zeta}.$$
Thus the required $\alpha$-bound depends on values and derivatives of $A$, $f$ and $g$. 
Finally, we set $\epsilon_0=\frac{1}{4C}\min\left\{1,C_1+C_2+C_3\right\}.$
\end{proof}
\begin{proof}[Proof of Theorem 1]
We apply the Theorem 3 for $\tilde{\mathcal{F}}$. Let us set
$$C=\max\{\left\|f\right\|_{1,\alpha},2\cdot\left\|d_f\tilde{\mathcal{F}}^{-1}\right\|\}$$ and let $\epsilon_0>0$ be such as in Lemma \ref{lemma5}. The set of invertible operators is open. Thus there is $\epsilon_1>0$ such that for $\left\|f-g\right\|_{1,\alpha}<\epsilon_1$ the inverse $d_g\tilde{\mathcal{F}}^{-1}$ exists and its norm is bounded by $C$. For $\epsilon=\min\{\epsilon_0,\epsilon_1\}$ the operator $\tilde{\mathcal{F}}$ satisfies $i)$ and $ii)$ in Theorem 3, hence there is $t_0>0$ depending on $C,$ $\epsilon$ and $V$ such that given a $J$-holomorphic variational vector $V$ field along $f$ the family 
$$f_t=\tilde{\mathcal{F}}^{-1}\left(\tilde{\mathcal{F}}(f)+td_f\tilde{\mathcal{F}}(V)\right)$$
is well defined and $J$-holomorphic for $|t|<t_0$. Moreover, $f_0=f$. 
\end{proof}
\noindent It is worth noting that the size of the neighborhood on which the Implicit function theorem can be applied actually depends on the $(1,\alpha)$-norm of $f$ and the operator norm of $d_f\tilde{\mathcal{F}}^{-1}$. Indeed, this two norms provide the appropriate constants $C$ and $\epsilon_0$ in lemmas \ref{lemma4} and \ref{lemma5}. Therefore, we have to uniformly bound them if we want to apply Theorem 1 for more than one $J$-holomorphic disc. We will do that in the next section when dealing with discs $f_r(\zeta)=f(r\zeta),\; r\in[\frac{1}{2},1]$. However, this case is trivial since one can rescale the integral operator $T$ an rely on the fact that $[\frac{1}{2},1]$ is compact.
  
\section{Proof of Theorem 2}
We start by proving the following lemma and a corolarry of Theorem 1.
\begin{lemma}\label{varia-lem}
Let $\phi=a+ib\colon \mathbb{D}\to \mathbb{C}$ be a standard holomorphic map. If $f\in\mathcal{C}^{2,\alpha}(\overline{\mathbb{D}})$ is a $J$-holomorphic disc, then 
$$V=\phi\cdot f':=af'+bJ(f)f'\in \mathcal{C}^{1,\alpha}(\overline{\D})$$
is a $J$-holomorphic variational vector field along $f$.
\end{lemma}
\begin{proof}
Recall from (\ref{VF}) that $f_y=J(f)f'$ and that the equation defining a $J$-holomorphic variational vector field $V$ along $f$ can be written in the form
$$D(V)=V_x+J(f)V_y+d_fJ(V)f_y=0.$$
Note that both, $f_x$ and $f_y$, are $J$-holomorphic variational vector fields along $f$ and therefore we have $D(f_x)=D(f_y)=0$. Together with $a_x=b_y$ and $a_y=-b_x$ this yields that 
$$D(af_x+bf_y)=a_xf_x + a_yJ(f)f_x+ b_xf_y+b_y J(f)f_y=0.$$
Since $J(f)^2=-id$ this completes the proof.
\end{proof}

\begin{coro}\label{varia-thm}
Let $f\in \mathcal{C}^{1,\alpha}(\overline{\D})$ be a $J$-holomorphic disc and let $V\in\mathcal{C}^{1,\alpha}(\overline{\D})$ be a $J$-holomorphic variational vector field along $f$ and satisfying $V(0)=0$. There exists a family $(f_t)\subset \mathcal{C}^{1,\alpha}(\overline{\D})$ of $J$-holomorphic discs such that $f_0=f$, $f_t(0)=f(0)$, $\frac{\mathrm{d}}{\mathrm{d}t}_{|_{t=0}}f_t=V$ and $f_t'(0)=f'(0)+tV'(0).$ 
\end{coro}
\begin{proof} As said, this is just a refined version of the Theorem 1. Indeed, as mentioned, for $J(f)=J_{st}$ we have $f'(0)=f_x(0)=f_{\zeta}(0)$. Hence we can repeat the same proof by using the following normalization of the Cauchy Green operator
$$T_0(f)(\zeta)=T(f)(\zeta)-T(f)(0)-\zeta\left[T(f)\right]_\zeta(0).$$
This yields that $\mathcal{F}$ defined as in (\ref{OPF}) but with $T_0$ satisfies $\mathcal{F}(f)(0)=f(0)$ and $[\mathcal{F}(f)]'(0)=f'(0)$. Moreover, it can be made locally invertible again by calculating a new adjoint operator of $d_f\mathcal{F}$ and finding $h_1,\ldots,h_N$ that span the complement of $\textrm{Range}(d_f\mathcal{F})$ and satisfy $h_j(0)=0$ and $h_j'(0)=0$.   
\end{proof}

\begin{proof}[Proof of Theorem \ref{main-thm}]
We follow the proof in \cite{gau-joo}. Assume that $f$ is not proper. Then there exists an open interval $P\subset \partial\D$ such that $f(P)\cap \partial\Omega=\emptyset.$ The idea is to shrink the domain of $f$ in order to asure that its perturbations will stay in $\Omega$ but, on the other hand, enlarge their values on $P$ so that they will contradict the extremality of $f$. 

Given $r\in[\frac{1}{2},1]$ we define $f_r(\zeta):=f(r\zeta)$. Let $K\subset\Omega$ be a compact set such that $f_r(\zeta)\in K$ for every $\zeta\in P$ and every $r\in[\frac{1}{2},1]$. For $R>0$ that will be fixed later we consider a smooth function $\chi_R\colon \partial\D\to \R$ compactly supported in $P$ and such that $\chi_R\equiv R$ on a slightly smaller closed subinterval  $P_1\subset P$. Using the Poisson kernel we construct a standard holomorphic function $\phi_R$ defined on $\D$ and continus up to the boundary with properties $\textrm{Im}(\phi_R)(0)=0$ and $\textrm{Re}(\phi_R) = \chi_R$ on $\partial\D.$ By Lemma \ref{varia-lem} the vector function $$V_r^R(\zeta):=(\zeta\;\textrm{exp}(\phi_R(\zeta))\cdot f_r'(\zeta)\in\mathcal{C}^{1,\alpha}(\overline{\D})$$ 
is a $J$-holomorphic variational vector field along $f_r$. Thus, by Corollary~\ref{varia-thm}, there exists a family $(h_{r,t}^R)_t$ of $J$-holomorphic discs such that
$$
\left\{
\begin{array}{l}
h_{r,0}^R(0) =  f_r(0)=f(0)\\
 \\
h_{r,t}^{R}\, ' (0)=f_r'(0)+t\;V_r^R\,'(0)=r(1+t\;\mathrm{exp}(\phi_R(0)))f'(0).
\end{array}
\right.
$$ 
Note that one can assume that these families are defined for $|t|<t_0(R)$. That is, this bound does not depend on $r\in[\frac{1}{2},1]$. Indeed, using a rescaling argument for the family $f_r$, the constants $C$ and $\epsilon_0$ in Lemma \ref{lemma4} and Lemma \ref{lemma5} can be chosen uniformly for all $f_r,$ $r\in[\frac{1}{2},1].$

Let $l$ denote the length of $P_1$. Then $\mathrm{exp}(\phi_R(0))\geq \mathrm{exp}(lR/2\pi)$. Hence if 
$$t>\frac{1-r}{r}\textrm{exp}(-lR/2\pi):=t_1(r,R)$$
we have
\begin{equation*}
(h_{r,t}^{R})'(0)=\lambda f'(0) \ \ {\rm for \ some}\ \lambda>1.
\end{equation*}
That is, if this holds for some $0<t<t_0(R)$ such that $h_{r,t}^R(\D)\subset\Omega$ this yields a contradiction to the extremality of $f$. Hence, let us seek the condition under which the image of $h_{r,t}^R$ is in $\Omega$.

First note that $\rho\circ f_r$ is negative on $\D$. Hence, there is $d(K,\rho)>0$ such that for $r\in[\frac{1}{2},1]$ and $\zeta\in P$ we have  
$$\rho\circ f_r(\zeta)\leq -d(K,\rho).$$
Moreover, there exist positive constants $C_1(R)>0$ and $t_1(R)>0$ such that for every $|t|<t_1(R)$ and every $r\in [\frac{1}{2},1]$ we have 
$$|\rho\circ h_{r,t}^R(\zeta)-\rho\circ f_r(\zeta)|\leq C_1(R) |t|.$$
Hence, if $|t|<\textrm{min}\left\{t_1(R),d(K,\rho)/C_1(R)\right\}$, we have $h_{r,t}^R(\zeta)\in\Omega$ for $\zeta\in P$.

Next, due to the subharmonicity of $\rho\circ f_r$ there is $C_2>0$ depending only on $\rho$ and $f(0)$ such that for every $\zeta\in\D$ and every $r\in[0,1)$ we have
$$\rho\circ f_r(\zeta)\leq -C_2(1-|\zeta|)<-C_2(1-r).$$
Since $\textrm{Re}(\phi_R)\equiv 0$ on $\partial\D\backslash P$ there is a constant $C_3>0$ independent of $R>0$ and a constant $t_2(R)>0$ such that for every $r\in[\frac{1}{2},1]$, $\zeta\in\partial\D\backslash P$  and $|t|<t_2(R)$ we have:
$$
|\rho\circ h_{r,t}^R(\zeta)-\rho\circ f_r(\zeta)|\leq C_3 |t|.
$$
Therefore, if $|t|\leq C_2(1-r)/C_3:=t_3(r)$ and $\zeta\in\partial\mathbb{D}\setminus P$ then $h_{r,t}^R(\zeta)\in\Omega$. 

Finaly, we fix $R>0$ such that
$$\exp\left(-\frac{lR}{2\pi}\right)<\frac{C_2}{2C_3}.$$ 
Then $t_1(r,R)<t_3(r)$. Moreover, we can take $r$ sufficiently close to $1$ so that $t_3(r)<\textrm{min}\left\{t_0(R),t_1(R),d(K,\rho)/C_1(R),t_2(R)\right\}.$ Thus for such $R$ and a parameter $t$ satisfying $t_1(r,R)<t<t_3(r)$ the disc $h_{r,t}^R$ is well defined, maps $\mathbb{D}$ into $\Omega$ and yields the desired contradiction.  
\end{proof}

\noindent\textbf{Final remark.} It was proved in \cite[Theorem 2.31]{PANG} that every $\mathcal{C}^2$-closed extremal disc mapping into a strongly pseudoconvex domain vanishes certain first order variations determined by the boundary function. Based on this property, a family of $J$-stationary discs was introduced in \cite{gau-joo}. As explained earlier, when dealing with non-integrable structures, such a necessarry condition for extremal discs is more appropriate than the geometric one provided by Tumanov \cite{TUMANOV}. In particular, the authors provide an explicit example of a stationary disc which admits such a required meromorphic lift but is not extremal. In contrast, using the variational approach, every extremal disc is $J$-stationary. However, as mentioned, their theory is restricted to small perturbations of $J_{st}$ for which Theorem $1$ can be obtained from the holomorphic case by using the Implicit function theorem. That is, this paper provides the main technical tool needed in order to remove this restriction.
\vskip 0.2 cm
\noindent  {\it Acknowledgments.}  Research of the author was supported in part by grants P1-0291, J1-9104 and BI-US/19-21-108 from ARRS, Republic of Slovenia. He would also like to thank Herve Gaussier and Franc Forstneri\v{c} for their useful comments while he was preparing the paper.

\end{document}